\documentclass[11pt,a4paper,fleqn]{scrartcl}    

\usepackage[english]{babel}                     
\usepackage[T1]{fontenc}                   
\usepackage{graphicx,comment}    
\usepackage{tikz}        
\usepackage{pgfplots}
\usepackage{etoolbox}

\pgfplotsset{compat=1.15}
\usepackage{mathrsfs}
\usetikzlibrary{arrows}           %
\usepackage[font=small]{quoting}      

\usepackage{enumitem,wrapfig}
\usepackage[colorlinks=true,linkcolor=blue,citecolor=red]{hyperref}
\usepackage{xcolor}
\usepackage{pgf,tikz}
\usepackage{picture}
\usepackage{amsmath,amsthm,amsfonts,amssymb}

\newtheorem{theorem}{Theorem}[section]

\newtheorem{prop}{Proposition}[section]

\theoremstyle{definition}
\newtheorem{definiz}{Definition}[section]

\newtheorem{rem}{Remark}[section]

\def\bk{\color{black}}

\newcommand{\ds}{\displaystyle}

\newcommand{\R}{\mathbb R}

\newcommand{\de}{\partial}
\newcommand{\eps}{\varepsilon}

\makeatletter
\def\@fnsymbol#1{%
   \ifcase#1\or
   \TextOrMath\textasteriskcentered *\or
   \TextOrMath{\textasteriskcentered\textasteriskcentered}{**}\or
   \TextOrMath \textdagger \dagger\or
   \TextOrMath \textdaggerdbl \ddagger \or
   \TextOrMath \textsection  \mathsection\or
   \TextOrMath \textparagraph \mathparagraph\or
   \TextOrMath \textbardbl \|\or
   \TextOrMath {\textdagger\textdagger}{\dagger\dagger}\or
   \TextOrMath {\textdaggerdbl\textdaggerdbl}{\ddagger\ddagger}\else
   \@ctrerr \fi
}
\makeatother

\begin{document}

\title{Some remarks on optimal insulation \\ with Robin boundary conditions}
\author{
Francesco Della Pietra\thanks{Universit\`a degli studi di Napoli Federico II, Dipartimento di Matematica e Applicazioni ``R. Caccioppoli'', Via Cintia, Monte S. Angelo - 80126 Napoli, Italia. Email:
f.dellapietra@unina.it} \and
Francescantonio Oliva\thanks{``Sapienza'' Universit\`a di Roma, Dipartimento di Scienze di Base e Applicate per l'Ingegneria, Via Scarpa 16, 00161 Roma, Italia. Email: francescantonio.oliva@uniroma1.it}
}
 
\date{}

\maketitle

\begin{abstract}
\noindent \textbf{Abstract}. We consider an optimal insulation problem of a given domain in $\R^{N}$. We study a model of heat trasfer determined by convection; this corresponds, before insulation, to a Robin boundary value problem. We deal with a prototype which involves the first eigenvalue of an elliptic differential operator. Such optimization problem, if the convection heat transfer coefficient is sufficiently large and the total amount of insulation is small enough, presents a symmetry breaking. \\

	
\end{abstract}

\begin{center}
\begin{minipage}{.8\textwidth}
\tableofcontents
\end{minipage}
\end{center}

\section{Introduction}

Thermal insulation, a field of study that has spanned multiple disciplines for centuries, has experienced a resurgence of interest in recent times. As environmental concerns have escalated and energy efficiency has become a paramount global objective, the importance of thermal insulation in reducing energy consumption and mitigating climate change has been increasingly recognized. While physics and engineering predominantly explore the development of innovative materials and technological advancements to enhance thermal insulation, the mathematical community has concurrently delved into a fascinating area of research: shape optimization. At the heart of this mathematical inquiry lies the problem of determining the optimal geometric configuration of an insulating material to minimize heat transfer.

One specific aspect of this problem, explored in depth in \cite{ab}, focuses on the thermal insulation of a solid body subject to conductive heat transfer with its surroundings (modeled by Dirichlet boundary conditions). In particular the aforementioned work generalizes previous research (see for example \cite{bcf,cf}) by considering a more general setting.

In the quoted papers the family of functionals that have been considered is 
\[
G_{\eps}(v,h) =\frac{1}{2} \int_{\Omega}|\nabla v|^2 d x+\frac{\varepsilon}{2} \int_{\Sigma_{\varepsilon}}|\nabla v|^2 d x-\int_{\Omega} f v \,d x
\]
in the Sobolev space $H_0^1\left(\Omega \cup \Sigma_{\varepsilon}\right)$, where $\Sigma_{\varepsilon}$   is a thin layer which features a variable thickness profile 
surrounding the boundary $\de\Omega$:
\begin{equation}\label{sigmaeps}
\Sigma_{\varepsilon}=\{\sigma+t \nu(\sigma): \sigma \in \partial \Omega, 0 < t<\varepsilon h(\sigma)\},
\end{equation}
 where $h$ gives the insulating profile and $\nu$, here and elsewhere, denotes the outer normal to $\de\Omega$. The temperature distribution, denoted by $u$, within the conducting body $\Omega$ is determined by minimizing the functional $G_\eps$ over the Sobolev space $H_0^1\left(\Omega \cup \Sigma_{\varepsilon}\right)$. This is equivalent to study a partial differential equation that incorporates the effects of heat sources, represented by $f\in L^{2}(\Omega)$, and the insulating properties characterized by the distribution $h$. More precisely, $u$ solves
\[
\begin{cases}-\Delta u=f & \text { in } \Omega, \\ -\Delta u=0 & \text { in } \Sigma_{\varepsilon}, \\ u=0 & \text { on } \de\Omega_\varepsilon, 
	\\ 
\displaystyle \frac{\partial u^{-}}{\partial \nu}=\varepsilon \frac{\partial u^{+}}{\partial \nu} & \text { on } \partial \Omega,\end{cases}
\]
where
\[
\Omega_\eps := \overline\Omega \cup \Sigma_\varepsilon
\]
is the union between the thermally conductive body and the insulating material. Here $\frac{\partial u^{-}}{\partial \nu}$ and $\frac{\partial u^{+}}{\partial \nu}$ denote, respectively, the normal derivatives of $u$ from inside and outside $\Omega$. As $\eps$ goes to $0$, then (\cite{ab}) $G_{\eps}$ $\Gamma-$converges  to 
\[
G(v,h)= \frac{1}{2} \int_{\Omega}|\nabla v|^2 d x+\frac{1}{2} \int_{\partial \Omega} \frac{v^2}{h} d\mathcal H^{N-1} -\int_{\Omega} f v dx
\]
defined in $H^{1}(\Omega)$ and whose minimizers $u$ are solutions to
\[
\begin{cases}-\Delta u=f & \text { in } \Omega, \\[.2cm] h \ds\frac{\partial u}{\partial \nu}+u=0 & \text { on } \partial \Omega.
\end{cases}
\]
 If one looks for the optimal configuration among all functions $h$ with fixed mass 
\begin{equation}
\label{condmintro}
\int_{\de\Omega}hd\mathcal H^{N-1}=m>0,
\end{equation} then can be shown (see \cite{b,bbn}) that the best distribution is realized by
\[
h_{opt}(\sigma)=\frac{m |u(\sigma)|}{\ds\int_{\de\Omega}|u|d\mathcal H^{N-1}},  \sigma\in \partial\Omega,
\]
where $u$ minimizes the functional
\[
\overline{G}(v) = \frac{1}{2} \int_{\Omega}|\nabla v|^2 d x+\frac{1}{2 m}\left(\int_{\partial \Omega}|v| d \mathcal{H}^{N-1}\right)^2-\int_{\Omega} f v d x.
\]
 On the other side, when heat is transferred to the outside through convection, which is a major mode of heat transfer, then Robin-type boundary conditions are best suited to the problem. Imagine, for example, the heat exchange that occurs on the surfaces of objects like a boiler, a cup of tea, or a building. In the spirit of the preceding discussion, this leads to study functionals of the type
\[
G_{\varepsilon,\beta}(v,h)=\frac{1}{2} \int_{\Omega}|\nabla v|^2 d x+\frac{\varepsilon}{2} \int_{\Sigma_{\eps}}|\nabla v|^2 d x+\frac{\beta}{2} \int_{\partial \Omega_{\eps}} v^2 d \mathcal{H}^{N-1}-\int_{\Omega} f v d x
\]
where $\beta>0$ is a fixed parameter. Minimizers of $G_{\varepsilon,\beta}$ are solutions to
\[
\begin{cases}-\Delta u_{\varepsilon}=f & \text { in } \Omega,\\ 
-\Delta u_{\varepsilon}=0 & \text { in } \Sigma_{\varepsilon},\\ 
\displaystyle \frac{\partial u_{\varepsilon}}{\partial \nu}+\beta u_{\varepsilon}=0 & \text { on } \partial \Omega_{\varepsilon}, 
\\[.2cm] 
\displaystyle \frac{\partial u_{\varepsilon}^{-}}{\partial \nu}=\varepsilon \frac{\partial u_{\varepsilon}^{+}}{\partial \nu} & \text { on } \partial \Omega.\end{cases}
\]
If we consider the asymptotics of $G_{\varepsilon,\beta}$ as $\varepsilon \rightarrow 0$, in \cite{dpnst} it is proved that, under suitable hypotheses of the regularity of $\partial \Omega$ and $h$, the functionals $G_{\varepsilon,\beta}$ $\Gamma$-converge to
\[
G_{\beta}(v,h)=\frac{1}{2} \int_{\Omega}|\nabla v|^2 d x+\frac{\beta}{2} \int_{\partial \Omega} \frac{v^2}{1+\beta h} d \mathcal{H}^{N-1}-\int_{\Omega} f v d x,
\]
with respect to the $L^2\left(\mathbb{R}^n\right)$ topology.

In contrast to the simpler case of homogeneous Dirichlet conditions, determining the optimal $h$ is a more challenging problem which, for the Robin case, has been addressed in \cite{dpnst}. In particular, it has been proved the existence of a couple $(u,h_{opt})$, with $u\in H^{1}(\Omega)$ and $h_{opt}$ given by
\begin{equation*}
	h_{opt}(\sigma):= \begin{cases}
		\displaystyle \frac{u(\sigma)}{c_u \beta} - \frac{1}{\beta} \ & |u(\sigma)|\ge c_u,
		\\
		0  & \text{otherwise},
	 \end{cases}
\end{equation*}
where $\sigma\in \partial\Omega$ and $c_u$ is the unique positive constant satisfying
\begin{equation}
\label{costantec}
c_u= \left(\frac{1}{|\{|u|\ge c_u\}| +\beta m}\right) \int_{\{|u|\ge c_u\}} |u| \ d\mathcal{H}^{N-1},
\end{equation}
which minimizes $G_{\beta}(v,h)$ among all the functions $v\in H^{1}(\Omega)$ and the functions $h\in L^{1}(\de\Omega)$ such that \eqref{condmintro} holds.

Let us stress that by $\{|u|\ge c_u\}$ we mean the set $\{\sigma\in \de\Omega\colon u(\sigma)\ge c_u\}$, while $|\{u\ge c_u\}|$ denotes its $\mathcal H^{N-1}$ Hausdorff measure.

Similar problems and related analysis have been addressed in \cite{acnt,bnnt,ck,dpNT,den,hll,hlly}.

\medskip

In this paper, we deal with the case of the eigenvalue problem under Robin boundary conditions. Indeed, let us consider operators as
\[
\langle \mathcal Av,\varphi\rangle =\int_{\Omega}\nabla v\cdot\nabla \varphi\, dx+\beta\int_{\de\Omega} \frac{v\varphi}{1+\beta h}d\mathcal H^{N-1}
\]
and the corresponding heat equation
\[
\begin{cases}
\de_{t}u+\mathcal A u =0,\\
u(0,x)=u_{0}(x).
\end{cases}
\]
 Then it is well known that, as $t$ goes to infinity,  the behavior of the temperature $u(t,x)$ mainly depends on the first eigenvalue of $\mathcal A$, that is
\[
\lambda(h)=\inf_{v\in H^{1}(\Omega)\setminus\{0\}}  \frac{\displaystyle\int_\Omega |\nabla v|^2dx + \beta\int_{
			\partial \Omega} \frac{v^2}{1+\beta h} \ d \mathcal{H}^{N-1}}{\displaystyle \int_\Omega v^2 dx},
\]
which is the main focus of the current paper. In this case, the problem of best insulation we consider is given by the following minimization:
\[
\lambda_{m}=\min_{h} \lambda(h)
\]
among all nonnegative functions $h$ with fixed mass $m$.

 Let us summarize the content of the paper. In Section 2 we rigorously define the setting of our problem, showing that a $\Gamma-$convergence result for our functional holds. Then, in Section 3, we address the study of $\lambda_{m}$. In particular, we prove that the value $\lambda_{m}$ is achieved also providing a characterization for the minimizers. Moreover, we prove some monotonicity and continuity properties for $\lambda_{m}$.  Finally, in Section 4, we show that, under suitable conditions on $\beta$ and $m$, the first eigenfuction, as well as the profile functions $h$, are not radial when $\Omega$ is a ball.

\section{Setting of the problem and $\Gamma-$convergence}

Assume that $\Omega\subset \R^{N}$ is an open bounded set with $C^{1,1}$ boundary and, just for this section, we let $h:\partial\Omega \mapsto \mathbb{R}$ to be a bounded positive Lipschitz function.  Then, for $\beta>0$, the problem of thermal insulation is related to the following functional
\begin{equation}\label{Feps}
	F_\varepsilon(v,h):=  \frac{\displaystyle\int_\Omega |\nabla v|^2 dx + \frac{\varepsilon}{2}\int_{\Sigma_\varepsilon} |\nabla v|^2 dx+ \frac{\beta}{2}\int_{
		\partial \Omega_{\varepsilon}} v^2 \ d \mathcal{H}^{N-1}}{\displaystyle \int_\Omega v^2 dx}, \  v\in H^1(\Omega_\varepsilon), v\not\equiv 0,
\end{equation}   
where \[
\Omega_\eps := \overline\Omega \cup \Sigma_\varepsilon
\] and $\Sigma_{\eps}$ is given by \eqref{sigmaeps}. Let us underline that the functional $F$ can be thought as defined in $L^2(\mathbb{R}^N)$ by requiring $F_\varepsilon(v,h) = \infty$ if $v\in L^2(\mathbb{R}^N) \setminus H^1(\Omega_\varepsilon)$.  

It is not difficult to show that, for any fixed $h$, the following problem
\[
	\tilde{\lambda} = \inf_{v\in H^1(\Omega_\varepsilon)} F_\varepsilon(v,h)
\]
admits a nonnegative minimizer $u_\varepsilon$ satisfying
\begin{figure}
\begin{tikzpicture}
\hspace{-6cm}\includegraphics[scale=3]{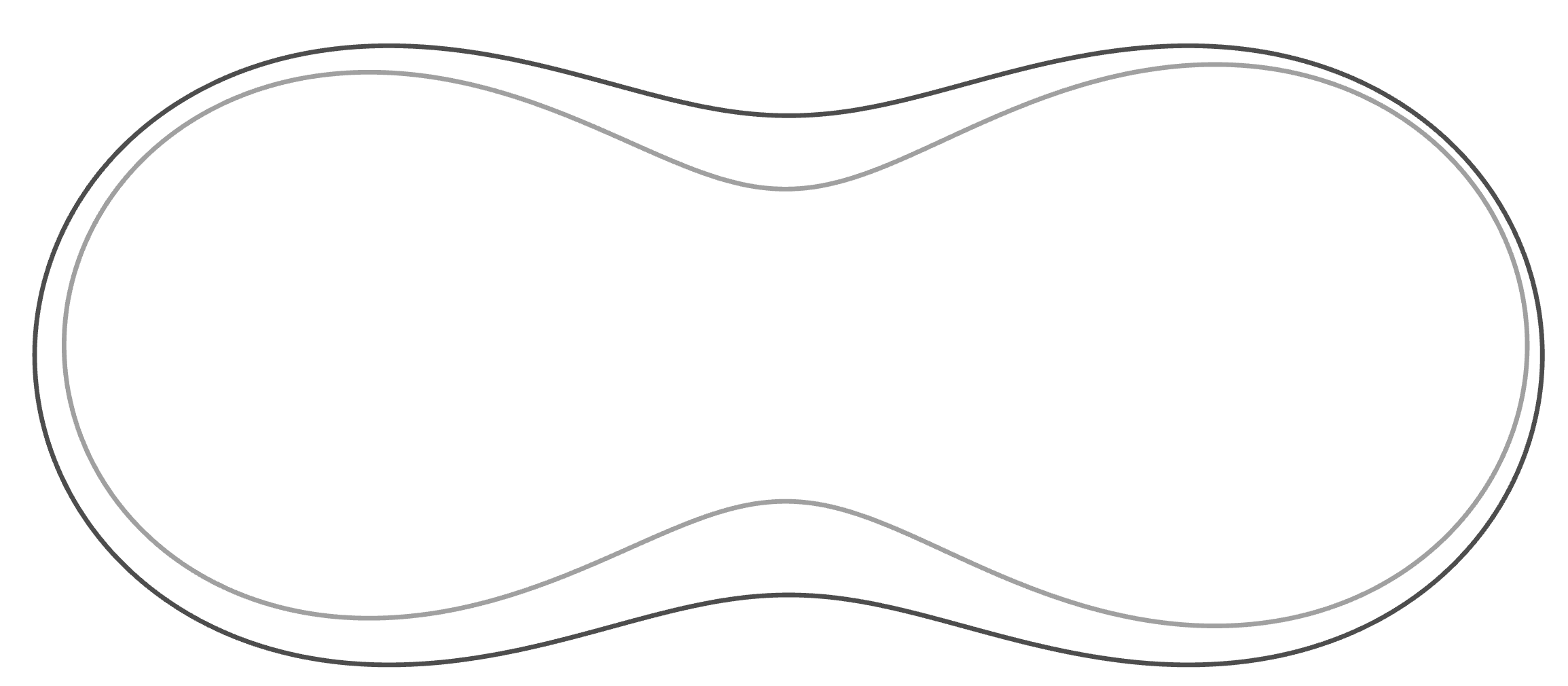}
\draw (-2.5,1.1) node[anchor=north west]{$\Omega$};
\draw (-3.6,.7) node[anchor=north west]{$\Sigma_\eps$};
\draw (-7.5,3.1) node[anchor=north west]{$-\Delta u_{\eps}=\tilde \lambda u_{\eps}$};
\draw (-5.5,1) node[anchor=north west]{\footnotesize $-\Delta u_{\eps}=0$};
\draw (-5.5,4.1) node[anchor=north west]{\footnotesize $\frac{\partial u_\varepsilon}{\partial \nu} + \beta u_\varepsilon = 0$};
\end{tikzpicture}
\end{figure}
\begin{equation}
	\begin{cases}
		\displaystyle - \Delta u_\varepsilon= \tilde{\lambda} u_\varepsilon \ &\text{in}\ \Omega,
		\\
		\displaystyle - \Delta u_\varepsilon= 0 \ &\text{in} \ \Sigma_\varepsilon,
		\\
		\displaystyle \frac{\partial u_\varepsilon}{\partial \nu} + \beta u_\varepsilon = 0\  &\text{on}\ \partial\Omega_{\varepsilon},\vspace{0.1 cm}
		\\
		\displaystyle	\frac{\partial u^-_\varepsilon}{\partial \nu} = \varepsilon\frac{\partial u^+_\varepsilon}{\partial \nu}\  &\text{on}\ \partial\Omega,
		\label{pbminveps}
	\end{cases}
\end{equation}
where, as already said, $\frac{\partial u_{\eps}^{-}}{\partial \nu}$ and $\frac{\partial u_\eps^{+}}{\partial \nu}$ stand for the normal derivatives of $u_\eps$ from inside and outside $\Omega$, respectively. In particular, last equation of \eqref{pbminveps} represents a transmission condition across $\partial\Omega$. 

Firstly we note that, as $\varepsilon \to 0$, the functional $F_\varepsilon$ $\Gamma$-converges in $L^2(\mathbb{R}^N)$ to  
\begin{equation}\label{Fset}
	F(v,h):=  \frac{\displaystyle\int_\Omega |\nabla v|^2dx + \beta\int_{
			\partial \Omega} \frac{v^2}{1+\beta h} \ d \mathcal{H}^{N-1}}{\displaystyle \int_\Omega v^2dx}, \  v\in H^1(\Omega), v\not\equiv 0,
\end{equation}
where, once again, $F$ is defined in $L^2(\mathbb{R}^N)$ as $F(v,h)=\infty$ if $v\in L^2(\mathbb{R}^N) \setminus H^1(\Omega)$. 
For the sake of completeness, let us precisely set the notion of $\Gamma$-convergence we deal with:
\begin{definiz}\label{defconv}
	The functional $F_\varepsilon$ (defined in \eqref{Feps}) $\Gamma$-converges in $L^2(\mathbb{R}^N)$ to $F$ (defined in \eqref{Fset}) as $\varepsilon\to 0$ if  for any $v\in L^2(\mathbb{R}^N)\setminus\{0\}$, it holds:
	\begin{itemize}  	
		\item[i)] \textit{liminf inequality}:  for any sequence $v_\varepsilon \in L^2(\mathbb{R}^N)$ which  converges to $v$ in $L^2(\mathbb{R}^N)$ as $\varepsilon\to 0$, then
		\begin{equation*}\label{liminf}
			\liminf_{\varepsilon \to 0} F_\varepsilon(v_\varepsilon,h) \ge F(v,h);
		\end{equation*}
		\item[ii)] \textit{limsup inequality}: there exists a sequence $v_\varepsilon \in H^1(\mathbb{R}^N)$ which converges to $v$ in $L^2(\mathbb{R}^N)$ as $\varepsilon\to 0$ such that
		\begin{equation*}\label{limsup}
			\limsup_{\varepsilon\to 0} F_\varepsilon(v_\varepsilon,h) \le F(v).
		\end{equation*}
	\end{itemize}
\end{definiz}
The following $\Gamma$-convergence result holds.
\begin{theorem}\label{gammaconvteo}
	Let $\beta > 0$ then the functional $F_\varepsilon$ (defined in \eqref{Feps}) $\Gamma$-converges to the functional $F$ (defined in \eqref{Fset}) in the sense of Definition \ref{defconv}.  
\end{theorem}
\begin{proof}
The proof can be carried on similarly to the one of \cite[Theorem 3.1]{dpnst}).
\end{proof}


\section{The minimization problem related to $F$}
\label{sec_simmetria}

For a given a fixed (positive) amount $m$ of insulating material, we are interested into finding out the optimal distribution $h$.
Then we set the problem by defining  
\[\mathcal{H}_m(\partial\Omega):=\{h\in L^1(\partial\Omega): h\ge 0, \int_{\partial\Omega} h \ d \mathcal{H}^{N-1}=m\}.\]
 Let us explicitly stress that, in this section,  we only need $h$ to be merely integrable on $\partial\Omega$. We then consider the functional $F$ (defined in \eqref{Fset}) where $h\in \mathcal{H}_m(\partial\Omega)$, namely we deal with 
\begin{equation}\label{F}
F(v,h):=  \frac{\displaystyle\int_\Omega |\nabla v|^2dx + \beta\int_{
		\partial \Omega} \frac{v^2}{1+\beta h} \ d \mathcal{H}^{N-1}}{\displaystyle \int_\Omega v^2dx}, \  v\in H^1(\Omega), v\not\equiv 0, h\in \mathcal{H}_m(\partial\Omega)
\end{equation}
where $\beta>0$.
First observe that, by reasoning as in \cite[Proposition 4.2]{dpnst}, one has that for any fixed $h\in \mathcal{H}_m(\partial\Omega)$ there exists a positive constant $C$ such that
\[
\int_{\Omega}|\nabla v|^{2}dx+\beta\int_{\de\Omega}\frac{v^{2}}{1+\beta h} \ d \mathcal{H}^{N-1} \ge C \|v\|_{H^{1}(\Omega)}.
\]
Then standard arguments allow to deduce
\[
\lambda(h)=F(u,h) = \min_{v\in H^1(\Omega),v\not=0} F(v,h)\]
where $u$ is nonnegative and it satisfies 
\begin{equation}
	\begin{cases}
		\displaystyle - \Delta u= \lambda (h) u \ &\text{in}\, \ \ \Omega,
		\\
		\displaystyle \frac{\partial u}{\partial \nu} + \frac{\beta u}{1+\beta h} = 0\  &\text{on}\,\ \partial\Omega.
		\label{pbminv}
	\end{cases}
\end{equation}

Now, in order to deal with minimization problem for $F$ with respect to both variables $v$ and $h$, we define the following quantity 
\begin{equation}
	h_v(\sigma):= \begin{cases}
		\displaystyle \frac{|v(\sigma)|}{c_v \beta} - \frac{1}{\beta} \ & |v(\sigma)|\ge c_v,
		\\
		0  &0 \le |v(\sigma)| < c_v,
		\label{h}
	 \end{cases}
\end{equation}
where $\sigma\in \partial\Omega$ and $c_v$ is the unique positive constant satisfying
\begin{equation}
\label{costantec}
c_v= \left(\frac{1}{|\{|v|\ge c_v\}| +\beta m}\right) \int_{\{|v|\ge c_v\}} v \ d\mathcal{H}^{N-1}.
\end{equation}
 Let us recall that by $\{|v|\ge c_v\}$ we mean the set $\{\sigma\in \de\Omega\colon |v(\sigma)|\ge c_v\}$, and with $|\{|v|\ge c_v\}|$ we denote its $\mathcal H^{N-1}$ Hausdorff measure.

Hence, we are in position to apply the following result (see \cite[Proposition $4.1$]{dpnst}):
\begin{prop} 
\label{minhprop}
Let $\beta, m>0$, let $v \in L^2(\partial \Omega)$, and let $h_{v} \in L^2(\de \Omega)$ be the function defined by \eqref{h}. Then $h_{v} \in L^2(\de \Omega) \cap \mathcal{H}_m(\partial \Omega)$ minimizes 
\begin{equation}\label{minbordo}
\min _{\hat{h} \in \mathcal{H}_m(\partial \Omega)}  \int_{\partial \Omega} \frac{v^2}{1+\beta \hat{h}} d \mathcal{H}^{N-1}.
\end{equation}
\end{prop}
From the previous proposition one deduces that
\begin{equation*}\label{minvh}
	\min_{h\in \mathcal{H}_m(\partial\Omega)} \min_{\substack{v\in H^1(\Omega)\\ v\not=0}} F(v,h) = F(u,h_u).
\end{equation*}
Theorem \ref{existencethm} below shows that the following problem admits a minimum:
\begin{equation}\label{min}
	\lambda_m:=  \inf_{\substack{(v,h)\in H^1(\Omega)\times \mathcal H_{m}(\de\Omega)}} F(v,h).
\end{equation}
First of all, let us recall an useful result proved in \cite[Prop. 4.3]{dpnst}:
\begin{prop}\label{prop43}
Assume that the set $\Omega$ is connected and let $\beta>0, m>0$ be fixed. Then the functional 
\[
J(v, h)=\displaystyle\int_\Omega |\nabla v|^2dx + \beta\int_{
		\partial \Omega} \frac{v^2}{1+\beta h} \ d \mathcal{H}^{N-1}
\] satisfies the following convexity condition in $H^1(\Omega) \times \mathcal{H}_m(\partial \Omega)$,
\[
\frac{1}{2}\left[J\left(v_1, h_1\right)+J\left(v_2, h_2\right)\right]>J\left(\frac{v_1+v_2}{2}, \frac{h_1+h_2}{2}\right) \quad \forall\left(v_1, h_1\right) \neq\left(v_2, h_2\right) .
\]
\end{prop} 
\begin{theorem}
\label{existencethm}
Let $\beta, m>0$. Then there exists a couple $(u, h_u) \in H^1(\Omega) \times \mathcal{H}_m(\partial \Omega)$ which minimizes \eqref{min} and where $h_u$ is defined into \eqref{h}. Furthermore the couple $(u,h_u)$ is a solution to \eqref{pbminv}.
\end{theorem}

\begin{proof}
Let $\left(u_n, h_n\right) \in H^1(\Omega) \times \mathcal{H}_m(\partial \Omega)$ be a minimizing sequence of \eqref{min}; we can always assume that $u_{n}\ge 0$ and suppose that $h_n$ is smooth enough.  Furthermore, $u_n$ solves
\begin{equation}
\label{approx1}
\begin{cases}-\Delta u_n=\lambda(h_{n})u_{n} & \text { in } \Omega, \\ 
\displaystyle \left(1+\beta h_n\right) \frac{\partial u_n}{\partial \nu}+\beta u_n=0 & \text { on } \partial \Omega .\end{cases}
\end{equation}
Indeed, for a fixed $h_n$, we can consider the auxiliary minimum problem
\[
\min _{v \in H^1(\Omega)} F\left(v, h_n\right)
\]
whose minimizer, denoted by $\bar{u}_n$, exists and it solves \eqref{approx1}. Hence, without loss of generality, we can assume $u_n=\bar{u}_n$ and $\|u_{n}\|_{L^{2}(\Omega)}=1$. 

The sequence $u_{n}$ is bounded in $H^{1}(\Omega)$ with respect to $n$ and standard arguments give that there exists $u\in H^{1}(\Omega)$ such that
$u_{n}\rightharpoonup u$ weakly in $H^{1}(\Omega)$, $u_{n}\rightarrow u$ strongly in $L^{2}(\Omega)$ and, by compactness of the trace embedding, $u_{n}\rightarrow u$ strongly in $L^{2}(\de\Omega)$ as $n\to\infty$. In particular, it also holds that $\|u\|_{L^{2}(\Omega)}=1$.

Now we claim that 
\begin{equation}\label{claim}
\liminf_{n} c_{u_{n}} >0,
\end{equation}
where $c_{u_{n}}$ is the constant appearing in \eqref{costantec} and related to $v=u_{n}$. We reason by contradiction; in particular it holds 
\[
\liminf _{n \rightarrow \infty} c_{u_n} \geq \frac{1}{|\de\Omega|+m \beta} \int_{\partial \Omega} \liminf _{n \rightarrow+\infty} u_n \chi_{\left\{u_n \geq c_{u_n}\right\}} d \mathcal{H}^{N-1},
\]
and, assuming $\displaystyle \liminf_{n\to\infty} c_{u_{n}} = 0$, then $u_{n}\to 0$ $\mathcal H^{N-1}$-a.e. on $\de\Omega$.
However, as $u_{n}$ is a solution to \eqref{approx1}, one yields to
\[
\int_{\Omega}\nabla u_{n}\cdot \nabla \varphi dx + \beta \int_{\de\Omega} \frac{u_{n}\varphi}{1+\beta h_{n}}d\mathcal H^{N-1}= \lambda(h_{n})\int_{\Omega}u_{n}\varphi dx,\quad \forall \varphi \in H^{1}(\Omega).
\]
Then one can take $n\to\infty$ into the previous, obtaining that (recall that $0\le \frac{1}{1+\beta h_{n}}\le 1$)
\begin{equation}\label{weak}
	\int_{\Omega}\nabla u\cdot \nabla \varphi dx = \lambda'\int_{\Omega}u \varphi dx,\quad \forall \varphi \in H^{1}(\Omega),
\end{equation}
where $\lambda'=\displaystyle \lim_{n\to\infty}\lambda (h_{n})$. Therefore one has that:
\begin{enumerate}
\item if $\lambda'=0$, then $\varphi=u$ into \eqref{weak} gives that $u$ is null in $\Omega$ (recall that $u$ is zero $\mathcal{H}^{N-1}$-a.e. on $\de\Omega$);
\item if $\lambda'>0$, then $\varphi = 1$ into \eqref{weak} provides, as $u$ is nonnegative, that $u$ is null in $\Omega$.
\end{enumerate}
As $\|u\|_{L^{2}(\Omega)}=1$, one gains a contradiction and \eqref{claim} is shown to hold. 

Now observe that it follows from Proposition \ref{minhprop} that $\left(u_n, h_{u_n}\right)$ ($h_{u_n}$ is defined into \eqref{h}) is still a minimizing sequence. As $\displaystyle \liminf_{n\to\infty} c_{u_{n}}>0$, then $h_{u_n}$ converges in $L^2(\partial\Omega)$ to some function $h_u \in$ $L^2(\partial \Omega) \cap \mathcal{H}_m(\partial \Omega)$, it turns out that the couple $(u, h_u)$ is a minimizer of \eqref{min}.

\end{proof}

\begin{rem}\label{remcost}
	Let us stress that for a minimizer $(u,h_u)$ to \eqref{min} the set $\{\sigma\in \partial\Omega: u(\sigma)>c_u\}$ needs to have positive $\mathcal{H}^{N-1}$ measure. Indeed, assuming $0\le u \le c_u$ $\mathcal{H}^{N-1}$ almost everywhere on $\partial\Omega$, one yields to
	\[
	c_u=\left(\frac{1}{\left|\left\{u \geq c_u\right\}\right|+m \beta}\right) \int_{\left\{u \geq c_u\right\}}u \,d \mathcal{H}^{N-1} = \frac{c_u\left|\left\{u = c_u\right\}\right|}{\left|\left\{u = c_u\right\}\right|+m \beta},
	\]
	which can occur if and only if $c_u$ is null or equivalently $u=0$ $\mathcal{H}^{N-1}$ almost everywhere on $\partial\Omega$. This would be in contradiction with the Hopf Lemma as $u$ satisfies \eqref{pbminv}.   
\end{rem}

\begin{rem}
Let us also explicitly observe that the minimizers $(u,h_u)$ to \eqref{min} satisfy
\begin{equation}\label{pminnonrad0}
	\begin{cases}
		\displaystyle - \Delta u= \lambda_m u \ &\text{in}\, \ \ \Omega,
		\\
		\displaystyle \frac{\partial u}{\partial \nu} + \beta u= 0\  &\text{on}\,\ \partial\Omega \cap \{u< c_u\}, \vspace{0.1 cm}
		\\ 
		\displaystyle \frac{\partial u}{\partial \nu} + \beta c_u= 0\  &\text{on}\,\ \partial\Omega \cap \{u\ge c_u\}.
	\end{cases}
\end{equation}
\end{rem}

\subsection{Properties of $\lambda_{m}$}

In this section we show some properties related to $\lambda_m$, which is defined into \eqref{min}.

\begin{prop}
\label{minm}
Let $\beta, m>0$ and let $v \in L^2(\partial \Omega)$.
Then
\[
\min _{\hat{h} \in \mathcal{H}_m(\partial \Omega)}  \int_{\partial \Omega} \frac{v^2}{1+\beta \hat{h}} d \mathcal{H}^{N-1}= \min _{\hat{h} \in \mathcal{H}_q(\partial \Omega), q\le m}  \int_{\partial \Omega} \frac{v^2}{1+\beta \hat{h}} d \mathcal{H}^{N-1}.
\]
\end{prop}
\begin{proof}
Let $h_{v} \in \mathcal H_{m}(\de\Omega)$ be the minimizer to \eqref{minbordo} found in Proposition \ref{minhprop} and let $\hat h\in \mathcal H_{q}(\de\Omega)$ where $0<q\le m$. 	

Following line by line the proof of Proposition 4.1 in \cite{dpnst}, it holds that
\[
\psi'(t)\le -\beta c_{v}^{2}\int_{\de\Omega} h_{v} d\mathcal H^{N-1}+\beta c_{v}^{2}\int_{\de\Omega}\hat h d\mathcal H^{N-1}=\beta c_{v}^{2}(-m+q) \le 0,
\]
as $q\le m$ and where
\[
\psi(t):=\int_{\partial \Omega} \frac{v^2}{1+\beta(\hat{h}+t(h_{v}-\hat{h}))} d \mathcal{H}^{N-1}, \quad t\in[0,1].
\]
This concludes the proof.
\end{proof}
\begin{rem}
\label{minmrem}
Looking at the proof of Proposition \ref{minm}, we observe that $\psi'(t)$ is strictly negative when $q<m$. This means that the function
\[
\min _{\hat{h} \in \mathcal{H}_m(\partial \Omega)}  \int_{\partial \Omega} \frac{v^2}{1+\beta \hat{h}} d \mathcal{H}^{N-1}
\]
is strictly decreasing in $m$.
\end{rem}

As a consequence of the previous proposition, one gains that $\lambda_m$ enjoys some monotonicity properties with respect to $m$.
\begin{prop}
\label{decrescenza}
The minimum $\lambda_{m}$ of \eqref{min} is decreasing in $m> 0$.
\end{prop}
\begin{proof}
Let $m'>m$ and, for the sake of clarity, let us denote by $(u_{m},h_{m}(u_{m}))\in H^{1}(\Omega) \times \mathcal H_{m}(\de\Omega)$ the minimizer of $\lambda_{m}$. Then it follows by Proposition \ref{minm} and Remark \ref{minmrem} that it holds
\begin{multline*}
\lambda_{m}= F(u_{m},h_{m}(u_{m})) > \int_{\Omega} |\nabla u_{m}|^{2}dx+ \beta\int_{\de\Omega} \frac{u_{m}^{2}}{1+\beta h_{m'}(u_{m})} d\mathcal H^{N-1} = \\ = F(u_{m},h_{m'}(u_{m})) \ge \lambda_{m'},
\end{multline*}
where we denoted by $h_{m'}(u_{m})$ the minimizer of $\int_{\de\Omega} \frac{u_{m}^{2}}{1+\beta h} d\mathcal H^{N-1}$ with respect to $h\in \mathcal H_{m'}(\de\Omega)$. This concludes the proof.
\end{proof}

\begin{prop}\label{continuita}
The minimum $\lambda_{m}$ is continuous in $m\ge0$. 
\end{prop}
\begin{proof}
Let $m,\eps>0$ and let $(u_{m+\eps},h_{u_{m+\eps}})$ be a minimizer for $\lambda_{m+\eps}$ where $\|u_{m+\eps}\|_{L^{2}(\Omega)}=1$, namely 
\[
\lambda_{m+\eps}=F(u_{m+\eps},h_{u_{m+\eps}}).
\]
Let us denote by
\[
\tilde h= \frac{m}{m+\eps} h_{u_{m+\eps}} \in \mathcal H_{m}(\de\Omega), 
\]
as $h_{u_{m+\eps}} \in \mathcal H_{m+\eps}(\de\Omega)$. One clearly has that
\[
\lambda_{m+\eps}= F(u_{m+\eps},\tilde h) -\beta\int_{\de\Omega} u_{m+\eps}^{2}\left(\frac{1}{1+\beta \tilde h}-\frac{1}{1+\beta h_{u_{m+\eps}}}\right)d\mathcal H^{N-1}.
\]
Now observe that for the right-hand of the previous it holds
\begin{multline}\label{cont0}
0\le \int_{\de\Omega} u_{m+\eps}^{2}\left(\frac{1}{1+\beta \tilde h}-\frac{1}{1+\beta h_{u_{m+\eps}}}\right)d\mathcal H^{N-1} = \\ = \frac{\eps}{m+\eps} \int_{\de\Omega} u_{m+\eps}^{2} \frac{\beta h_{u_{m+\eps}}}{(1+\beta \tilde h)(1+\beta h_{u_{m+\eps}})}  d\mathcal H^{N-1} \le 
\\
\le \frac{\eps}{m} \int_{\de\Omega} \frac{u_{m+\eps}^{2}}{1+\beta  h_{u_{m+\eps}}} d\mathcal H^{N-1}
\end{multline}
and the right-hand of the previous is bounded in $\eps$ thanks also to Proposition \ref{decrescenza}; then we conclude that
\begin{equation}
\label{diseg1}
\lambda_{m+\eps}\ge  F(u_{m+\eps},\tilde h) - C\eps,
\end{equation}
for some positive constant $C$ independent of $\eps$. Hence, as $\tilde h \in \mathcal H_{m}(\de\Omega)$, by the variational definition of $\lambda_{m}$, \eqref{diseg1} and Proposition \ref{decrescenza} we get
\[
0\le \lambda_{m}-\lambda_{m+\eps} \le F(u_{m+\eps},\tilde h)- F(u_{m+\eps},\tilde h) +C \eps= C\eps,
\]
which proves the continuity in case $m>0$.

 Now, if $m=0$, then one can reason as for obtaining \eqref{cont0} in order to deduce 
\begin{equation}\label{cont1}
0\le \lambda_{0}-\lambda_{\eps} \le \int_\Omega u^2_\eps \displaystyle \frac{\beta^2 h_{u_\eps}}{1+\beta h_{u_\eps}} d\mathcal H^{N-1}.	
\end{equation}
Then, as $h_{u_\eps}$ tends to $0$ in $L^1(\partial\Omega)$ (recall $\|h_{u_\eps}\|_{L^1(\partial\Omega)}=\eps$) and as $u_\eps$ strongly converges in $L^2(\partial\Omega)$ as $\eps \to 0$, \eqref{cont1} concludes the proof. 
\end{proof}

\section{Symmetry breaking}

In this section we show a symmetry breaking phenomenon for the minimizer of \eqref{min} when the domain is a ball.


\begin{theorem}\label{teo_radial}
	Let $\Omega$ be a ball and let $(u, h_u)\in H^{1}(\Omega)\times \mathcal H_{m}(\de\Omega)$ be the nonnegative minimizer to \eqref{min} where $\beta>0$ and $h_u$ is defined in \eqref{h}. Then there exists a positive $\beta^*$ such that:
	\begin{enumerate}
	\item if $0< \beta<\beta^*$, then $u$ is radial for any $m>0$;
	\item if $\beta>\beta^*$, then there exists $\overline{m}(\beta)$ such that $u$ is not radial when $m<\overline{m}(\beta)$, while it is radial if $m> \overline{m}(\beta)$.
	\end{enumerate} 
\end{theorem}

\begin{rem}
As a immediate consequence of Theorem \ref{teo_radial}, one can observe that, if $\beta$ is large enough, the optimal density is not constant for a small amount of insulating material. Let also observe that, as $\beta \to \infty$, Theorem \ref{teo_radial} is consistent with the result proved in \cite[Theorem $3.1$]{bbn} (see also \cite{bbn2}).
\end{rem}
\begin{rem}\label{rem_posbordo}
	Here we stress that, if $u$ is a solution to \eqref{pminnonrad0} which is constant on $\partial\Omega$, one can show that $u>c_u$ on $\partial\Omega$; indeed this easily follows (see Remark \ref{remcost}) by the definition itself of $c_u$, which is given into \eqref{costantec}. Then \eqref{pminnonrad0} reads as 
	\begin{equation}\label{pminnonrad0bis}
		\begin{cases}
			\displaystyle - \Delta u= \lambda_m u \ &\text{in}\, \ \ \Omega,
			\\
			\displaystyle \frac{\partial u}{\partial \nu} = - \frac{\beta}{|\partial\Omega| +\beta m} \int_{\partial\Omega} u \ d\mathcal{H}^{N-1} \  &\text{on}\,\ \partial\Omega,
		\end{cases}
	\end{equation}
	which is a fact  widely  used throughout the proof of Theorem \ref{teo_radial}.
\end{rem}
\begin{proof}[Proof of Theorem \ref{teo_radial}]
First let us stress that, for the sake of simplicity and only in this proof, we keep track of the dependence of $\beta$ into $\lambda_m$. 
 
Let us observe that it follows from Proposition \ref{continuita} that 
\[\lambda_m(\beta) \stackrel{m\to 0^+}{\to} \lambda_0(\beta) = \lambda^R(\beta) := 
\min_{v\in H^1(\Omega), v\not=0} \frac{\displaystyle\int_\Omega |\nabla v|^2 dx + \beta\int_{
			\partial \Omega} v^2 \ d \mathcal{H}^{N-1}}{\displaystyle \int_\Omega v^2 dx},
\]
which is the Robin first eigenvalue; also observe that $\lambda^R(\beta)$ is actually increasing with respect to $\beta$ and it holds $0=\lambda^R(0)\le\lambda^R(\beta)< \lambda^D$ where
\[\lambda^D := \min_{v\in H^1_0(\Omega), v\not=0} \frac{\displaystyle\int_\Omega |\nabla v|^2 dx}{\displaystyle \int_\Omega v^2 dx}.\]
Moreover, as $m\to\infty$, one has that
\[\lambda_m(\beta) \stackrel{m\to \infty}{\to} \lambda_\infty = 0.\]
Finally, let us also denote by 
\begin{equation}\label{neumann}
	\lambda^N := \min_{v\in H^1(\Omega), \int_\Omega v =0, v\not= 0} \frac{\displaystyle\int_\Omega |\nabla v|^2 dx}{\displaystyle \int_\Omega v^2 dx}.
\end{equation}
Then one clearly has that  
\begin{equation*}\label{lambdaN<}
0=	\lambda_\infty<\lambda^N<\lambda^D
\end{equation*}
and, for any $m>0$ and   for any $\beta>0$, 
\begin{equation}\label{lambdam<}
0=	\lambda_\infty<\lambda_m(\beta)<\lambda^R(\beta)<\lambda^D.
\end{equation}
Now observe that there exists some positive $\beta^*$ such that $\lambda^R(\beta^*)= \lambda^N$ which means that, as \eqref{lambdam<} is in force and $\lambda^{R}(\beta)$ is increasing in $\beta$, it holds $\lambda_m(\beta)< \lambda^N$ if $\beta<\beta^*$ for any $m>0$. Moreover if $\beta>\beta^*$ there exists $\overline{m}(\beta)$ such that
\[\lambda_m(\beta)<\lambda_{\overline{m}(\beta)}(\beta) = \lambda^N \text{ for any } m>\overline{m}(\beta),\]
as both Propositions \ref{decrescenza} and \ref{continuita} are in force.\bk

In order to prove the Theorem we split the proof in two different cases.

\medskip

\textbf{Case i)}: $\beta > \beta^*$ and $m< \overline{m}(\beta)$ ($u$ is not radial).

\medskip

Under the assumption of Case i), as shown before, one has that $\lambda_m(\beta)> \lambda_N$. Here we work by contradiction, assuming that $u$, a minimizer to \eqref{min}, is radial.

\medskip

First observe that, as $u$ is radial, then it is also constant and positive on $\partial\Omega$; indeed, if one supposes that $u$ is null on $\partial\Omega$ then $\lambda_m(\beta) = \lambda^D$ which is contradiction with \eqref{lambdam<}. In this particular case, it follows from \eqref{costantec} that it needs to hold $u> c_u$ on $\partial\Omega$ (see also Remark \ref{rem_posbordo}).

\medskip

Now, for $\varepsilon>0$, we test the functional $F$ (defined in \eqref{F}) with $u+\varepsilon v$ where $v$ is the first eigenfunction of the Neumann problem (i.e. $v$ is a minimizer to \eqref{neumann}).
First observe that a simple computation takes to
\begin{multline*}
	(\lambda_m(\beta) -\lambda^N)\int_\Omega uv = -\int_{\partial\Omega} \frac{\partial u}{\partial \nu} v d\mathcal{H}^{N-1} =\\= \frac{ \beta}{|\partial\Omega| +\beta m} \left(\int_{\partial\Omega} u d\mathcal{H}^{N-1}\right) \left(\int_{\partial\Omega}  v d\mathcal{H}^{N-1}\right)=0
\end{multline*} 
where the last two equalities follow from \eqref{pminnonrad0bis} and from \eqref{neumann}, since we use that $\int_{\de\Omega}v\,d\mathcal H^{N-1}=0$.
This shows that $u$ and $v$ are orthogonal.
Also observe that, since $u$ is strictly positive, one obtains $u+\varepsilon v>c$ in $\partial\Omega$ for a positive constant $c$ and for $\varepsilon$ small enough. Then, denoting by $u_\varepsilon:= u+\varepsilon v$, one has that 
\[c_{u_\varepsilon}= \left(\frac{1}{|\{u_\varepsilon\ge c_{u_\varepsilon}\}| +\beta m}\right) \int_{\{u_\varepsilon\ge c_{u_\varepsilon}\}} u_\varepsilon \ d\mathcal{H}^{N-1},\]
which gives
\begin{equation*}\label{ceps}
\liminf_{\varepsilon \to 0}c_{u_\varepsilon}\ge  \left(\frac{1}{|\partial\Omega| +\beta m}\right) \int_{\partial\Omega} \liminf_{\varepsilon \to 0} u_\varepsilon \chi_{\{u_\varepsilon\ge c_{u_\varepsilon}\}}\ d\mathcal{H}^{N-1}.
\end{equation*}
Now observe that, if one supposes that $c_{u_\eps}$, up to subsequences, tends to zero as $\varepsilon \to 0$, then the same inequality implies that 
\begin{equation*}
	0 \ge  \left(\frac{1}{|\partial\Omega| +\beta m}\right) \int_{\partial\Omega} u \ d\mathcal{H}^{N-1},
\end{equation*}
which is a contradiction since $u$ is positive on $\partial\Omega$. Hence we have that, up to subsequences, $c_{u_\varepsilon}$ tends, as $\varepsilon \to 0$ and by   uniqueness, to $c_{u}$.
Now observe that, as Proposition \ref{minhprop} is in force, it holds 
\begin{equation*}
	F(u_\varepsilon, h_{u_\varepsilon}) \le F(u_\varepsilon, h_{u})
\end{equation*}
where $h_{u_\varepsilon}$ and $h_u$ are defined in \eqref{h}.
Then the previous inequality takes to
\begin{multline*}\label{radial0}
	\int_{\{u_\varepsilon<c_{u_\varepsilon}\}} u^2_\varepsilon d\mathcal{H}^{N-1} + \frac{1}{|\{u_\varepsilon\ge c_{u_\varepsilon}\}|+m\beta} \left(\int_{\{u_\varepsilon\ge c_{u_\varepsilon}\}} u_\varepsilon d\mathcal{H}^{N-1}\right)^2
	\\
	\le
 	\int_{\partial\Omega} \frac{u^2_\varepsilon}{1+h_u\beta} d\mathcal{H}^{N-1} .
\end{multline*}
Now, as $u>c_{u}$ on $\partial\Omega$ (see Remark \ref{rem_posbordo}), it follows that there exists $\overline{\varepsilon}$ such that $u_\eps >c_{u_\varepsilon}$ on $\partial\Omega$ for any $\varepsilon<\overline{\varepsilon}$.
Then one has
\begin{equation}\label{radial}
	\begin{aligned}
		\lambda_m(\beta) &\le F(u_{\varepsilon}, h_{u_{\varepsilon}}) =\frac{1}{1+\varepsilon^2}\left(\int_\Omega |\nabla u_\varepsilon|^2 dx +  \frac{\beta}{|\partial\Omega|+m\beta} \left(\int_{\partial\Omega} u_\varepsilon d\mathcal{H}^{N-1}\right)^2\right)
		\\
		& 
		=\frac{1}{1+\varepsilon^2}\left(\int_\Omega |\nabla u_\varepsilon|^2dx +  \frac{\beta}{|\partial\Omega|+m\beta} \left(\int_{\partial\Omega} u d\mathcal{H}^{N-1}\right)^2\right) = \frac{\lambda_m(\beta)+\varepsilon^2\lambda_N}{1+\varepsilon^2},
	\end{aligned}
\end{equation}
where we used that $u$ and $v$ are orthogonal,  $\int_{\partial\Omega} v \ d \mathcal{H}^{N-1}=0$ and, without loss of generality, we have also assumed $\int_\Omega u^2= \int_\Omega v^2 =1$.

Finally we observe that \eqref{radial} implies $\lambda_m(\beta) \le \lambda^N$ which is in contradiction with the assumptions assumed for this case. Therefore $u$ is not radial if $\beta>\beta^*$ and $m<\overline{m}(\beta)$.

\medskip

\textbf{Case ii)}: $0< \beta<\beta^*$ or  $\beta>\beta^*$ and $m> \overline{m}(\beta)$ ($u$ is radial).

In this case it holds $\lambda_m(\beta) < \lambda^N$ and we aim on proving that only radial solutions can exist.


\medskip

Now let us assume that a non-radial solution to \eqref{min} exists. Let $u(r,\omega)$ be such solution in polar coordinates. Moreover, by spherical symmetrization (see \cite{den,sper}), one can assume that there exists $\omega_0 \in \mathcal{S}^{N-1}$ such that $u(r,\omega)$ is spherically symmetric in the direction $\omega_0$ that is
\[
u(r,\omega_1) \ge u(r,\omega_2) \ \text{for all} \ 0<r<R, \ \text{if} \ |\omega_1-\omega_0| \le |\omega_2-\omega_0|. 
\] 
In particular, $u$ minimizes
\begin{multline}\label{minimosost}
\lambda_{m}=\\ \min_{v\in H^{1}(\Omega)} \frac{\ds\int_{\Omega}|\nabla v|^{2}dx+\beta \int_{\{|v|<c_{v}\}} v^{2}d\mathcal H^{N-1}+ \frac{\beta}{|\{|v|\ge c_{v}\}|+m\beta} \left(\int_{\{|v|\ge c_{v}\}}|v|\,d\mathcal H^{N-1}\right)^{2}}{\ds\int_{\Omega} v^{2}\,dx}.
\end{multline}
Indeed, by properties of the rearrangements, the gradient integral decreases under spherical symmetrization; on the other hand, the constant $c_{u}$ in \eqref{minimosost}, as well as the boundary integrals, does not vary under symmetrization.

The function $u$ is also a solution to
\begin{equation}\label{pminnonrad}
	\begin{cases}
		\displaystyle - \Delta u= \lambda_m u \ &\text{in}\, \ \ \Omega,
		\\
		\displaystyle \frac{\partial u}{\partial \nu} + \beta u= 0\  &\text{on}\,\ \partial\Omega \cap \{u< c_u\}, \vspace{0.1 cm}
		\\ 
		\displaystyle \frac{\partial u}{\partial \nu} + \beta c_u= 0\  &\text{on}\,\ \partial\Omega \cap \{u\ge c_u\}.
	\end{cases}
\end{equation}

Then let us test \eqref{pminnonrad} with $v$ which is the first nontrivial eigenfunction of the Neumann eigenvalue problem on $\Omega$. This yields to

\begin{equation*}
	\int_\Omega \nabla u\cdot \nabla vdx - \int_{\partial\Omega} \frac{\partial u}{\partial \nu} v \ d\mathcal{H}^{N-1} = \lambda_m\int_\Omega uvdx,
\end{equation*}
which takes to
\begin{equation}\label{contest}
	- \int_{\partial\Omega} \frac{\partial u}{\partial \nu} v \ d\mathcal{H}^{N-1} = (\lambda_m-\lambda^N)\int_\Omega uvdx.
\end{equation}
Now observe that one can choose $v$ to be spherically symmetric with respect to $\omega_0$ and antisymmetric with respect to the reflection about the orthogonal plane to $\omega_0$. This means that the integral on the right-hand side of \eqref{contest} is positive as $u$ is non-radial. 

Moreover one has that $\frac{\partial u}{\partial \nu}$ is spherically symmetric as well thanks to the second and third equations in \eqref{pminnonrad} since $u$ is spherically symmetric with respect to $\omega_0$.

This allows to reasoning as before deducing that the left-hand of \eqref{contest} is positive. Then one has $\lambda_m\ge \lambda^N$  which gives a contradiction to the assumptions of the current case. This proves that $u$ is radial and the proof is concluded.   
 
\end{proof}

\section*{Acknowledgements}
This work has been partially supported by PRIN PNRR 2022 ``Linear and Nonlinear PDEs: New directions and Applications'' and by GNAMPA of INdAM.

\end{document}